\documentclass[12pt]{amsart}

\usepackage{amsfonts,amsmath,latexsym,amssymb,verbatim,amsbsy}
\usepackage{amsthm}

\usepackage{pstricks}

\theoremstyle{plain}
\newtheorem{THEOREM}{Theorem}[section]

\newtheorem{theorem}[THEOREM]{Theorem}

\newtheorem{lemma}[THEOREM]{Lemma}

\theoremstyle{definition}

\theoremstyle{remark}

\newtheorem{remark}[THEOREM]{Remark}

\newtheorem{example}[THEOREM]{Example}

\newcommand{\thm}[1]{Theorem~\ref{#1}}
\newcommand{\lem}[1]{Lemma~\ref{#1}}

\newcommand{\rem}[1]{Remark~\ref{#1}}

\newcommand{\N}{\ensuremath{\mathbb{N}}}   
\newcommand{\Z}{\ensuremath{\mathbb{Z}}}   
\newcommand{\R}{\ensuremath{\mathbb{R}}}   
\newcommand{\T}{\ensuremath{\mathbb{T}}}   
\newcommand{\K}{\ensuremath{\mathbb{K}}}

\renewcommand{\S}{\ensuremath{\mathbb{S}}}

\def \d {\delta}

\def \e {\varepsilon}
\def \f {\varphi}

\def \k {\kappa}
\def \l {\lambda}
\def \L {\Lambda}

\def \n {\nabla}

\def \th {\theta}
\def \D {\Delta}

\def \O {\Omega}

\def \bs {\backslash}

\def \calC {\mathcal{C}}

\def \K {\mathcal{K}}
\def \calZ {\mathcal{Z}}
\def \U {\mathcal{U}}
\def \calD {\mathcal{D}}

\def \lan {\langle}
\def \ran {\rangle}
\def \p {\partial}
\def \ra {\rightarrow}
\def \ss {\subset}

\DeclareMathOperator{\supp}{supp} %
\DeclareMathOperator{\interior}{int} %
\DeclareMathOperator{\conv}{conv} %
\DeclareMathOperator{\diver}{div} %
\DeclareMathOperator{\dist}{dist} %
\DeclareMathOperator{\Rg}{Rg} %
\newcommand{\rest}[2]{#1\raisebox{-0.3ex}{\mbox{$\mid_{#2}$}}}

\begin{document}

\title[Convex integration]{Convex integration for a class of active scalar equations}

\author{R. Shvydkoy}
\thanks{The work was partially supported by NSF grant DMS -- 0907812. The author is grateful to S. Friedlander, F. Gancedo, V. \v{S}verak and V. Vicol for stimulating discussions.}
\address{Department of Mathematics, Stat. and Comp. Sci.\\
 851 S  Morgan St., M/C 249\\
        University of Illinois\\
        Chicago, IL 60607}
\email{shvydkoy@math.uic.edu}

\date{\today}

\begin{abstract}
We show that a general class of active scalar equations, including porous media and certain magnetostrophic turbulence models, admit non-unique weak solutions in the class of bounded functions. The proof is based upon the method of convex integration recently implemented for equations of fluid dynamics in \cite{ds1,spanish}. 
\end{abstract}



\maketitle

\section{Introduction}

We study weak solutions to a general class of non-dissipative active scalar equations
\begin{align}
\th_t + u \cdot \n \th & = 0, \label{ase}\\
\diver_x u & = 0, \label{div}\\
u & = T[\th] \label{int}
\end{align}
on the periodic domain $\T^n$ with zero mean condition:
\begin{equation}\label{mean}
\int_{\T^n} u dx = \int_{\T^n} \th dx = 0.  
\end{equation}
Here $T$ is a Fourier multiplier in the spacial direction only:
$$
\widehat{T[\th]}(\xi) = {m(\xi)}\hat{\th}(\xi), \quad \xi \in \Z^n.
$$
We make the following assumptions about $m$. We assume that $m: \R^n \bs \{0\} \ra \R^n$ is even, $0$-homogeneous, and $m(\xi) \cdot \xi = 0$ a.e. Note that $m$ is not necessarily smooth or even bounded, so the operator $T$ is not assumed to be of Calderon-Zygmund type.

We say that the pair $(\th,u)\in L^2_{loc}(\R \times \T^n)$ is a weak solution of the system \eqref{ase}--\eqref{mean} on $\R$ if for every $\phi \in C_0^\infty(\R \times \T^n)$ the following holds
$$
\int_{\R\times\T^n} \th ( \p_t \phi + u \cdot \n \phi) dx dt = 0,
$$
$T[\th]$ defines a distribution for a.e. $t\in \R$,  \eqref{div}, \eqref{int} hold in the distributional sense for a.e. $t\in \R$, and \eqref{mean} holds in the usual sense.

Well-posedness theory for this type of active scalar equations of course depends on the specifics of the non-local relation \eqref{int}. It is clear however that if $T:L^2 \ra L^2$ is bounded, the standard energy method applies to produce unique local in time solutions in $H^{n/2 + 1 + \e}$. As we will see the boundedness of $T$ is not always guaranteed even for some physically relevant examples (see below). 

The primary motivation for studying weak solutions of hydrodynamic equations comes from our attempt to understand the laws of turbulence in the limit of infinite Reynolds number. As implied by the classical K41 theory of Kolmogorov \cite{k41} and Onsager's conjecture \cite{onsager} the limiting solutions to the Navier-Stokes equations ought to have sharp Besov regularity $1/3$ and anomalous energy dissipation (see \cite{eyink,shv} for recent accounts). Active scalar equations are viewed both as a viable replacement to the fully equipped models based on Navier-Stokes and Euler equations, and as models of their own as in the case of recently proposed Moffatt's magnetostrophic turbulence of the fluid Earth's core \cite{moffatt}. Starting from Kraichnan's theory of anomalous scaling for passive scalars \cite{kr} subsequent more rigorous observations were made for instance in \cite{const} in the context of active scalars of the type we consider here. 

Although no deterministic solutions are known to exist that satisfy all the characteristics of a turbulent flow, first anomalous examples for the Euler equation were found by Scheffer \cite{scheffer1} and Shnirelman \cite{shn1,shn2}, where the velocity field $u\in L^2_tL^2_x$ has a finite support in time. Recently the spacial regularity was improved to $L^\infty$ in the groundbreaking works of De Lellis and  Sz{\'e}kelyhidi \cite{ds1,ds2}. The Euler equation there is viewed as a hyperbolic conservation law with constraints in the spirit of Tartar \cite{tartar}. Specifically, $\diver_{t,x} U = 0$, with $U(t,x) \in K$, where $U$ is an $(n+1)\times (n+1)$ matrix, and $K$ is a set of matrices that encodes the nonlinearity as a pointwise constraint. The method of convex integration is then used to "reach" the constraint $K$ by superimposing oscillatory plane wave solutions from within the more relaxed set $\interior\{ \conv (K)\}$ (see \cite{gromov,kms,ms,spring} for the origins and many notable applications of the method). Following a similar -- although in many technical details much different -- approach Cordoba, Faraco, and Gancedo \cite{spanish} showed the same result for the periodic 2D porous media equation, whereby proving non-uniqueness in the class of bounded weak solutions. 
The main objective of this work is to explore the extent to which the method of convex integration applies to construct "wild" weak solutions to the general class of active scalar equations \eqref{ase}--\eqref{mean}. 

To state our result we call a point $\xi \in \S^{n-1}$ regular if $m: \S^{n-1} \ra \R^n$ is a local $C^1$-immersion near $\xi$. 
\begin{theorem}\label{t:main}
Suppose that
\begin{itemize}
\item[(i)] $m(\xi)$ is even and $0$-homogeneous;
\item[(ii)] the range $\{ m(\xi):  \xi \in \S^{n-1} \text{ regular } \}$ spans $\R^n$.
\end{itemize}
Then there exists a weak solution $(\th,u)$ to \eqref{ase}--\eqref{mean} such that
\begin{enumerate}
\item $\th, u \in L^{\infty}(\R \times \T^n)$;
\item $\th = u = 0$ for all $t \not \in (0,T)$;
\item $|\th(x,t)| = 1$ for a.e.  $t\in(0,T)$ and $x \in \T^n$.
\end{enumerate}
\end{theorem}

Our proof employs elements, such as the use of persistent oscillatory waves and the $T4$-configurations,  already present in \cite{ds1,ms,scheffer2} and particularly in \cite{spanish}. We emphasize however that the integral relation between velocity and scalar \eqref{int} in general cannot be interpreted in the divergence form or as any first order linear differential relation suitable to the traditional Tartar's framework (see examples below). It is therefore necessary to treat \eqref{int} as a separate constraint to be satisfied along with the local pointwise constraint as described earlier. As a byproduct of our construction we show that the solutions may have Fourier support in narrow conical regions.

\begin{theorem} There exists two regular frequency vectors $\xi_1,\xi_2 \in S^{n-1}$ such that for any relatively open neighborhoods $W_1, W_2 \ss S^{n-1}$ of $\xi_1$ and $\xi_2$, respectively, there exists a weak solution $(\th,u)$ satisfying all of the properties stated in \thm{t:main} and in addition
\begin{equation}\label{}
\supp\{ \widehat{\th(t,\cdot)}, \widehat{u(t,\cdot)} \} \ss \R \cdot (W_1 \cup W_2)
\end{equation}
for almost every $t \in (0,T)$.
\end{theorem}
We will see that there are in fact infinitely many such distinct pairs $\xi_1$, $\xi_2$. Consequently, we can construct infinitely many mutually orthogonal solutions.

Examples of equations covered by \thm{t:main} arise in a variety of physical circumstances. We present a few of them here.
\begin{example} The 2D porous media equation governs evolution of the density $\rho$ of an incompressible fluid:
$$
\rho_t + u \cdot \n \rho = 0,
$$
with $\diver u = 0$, where $u$ and $\rho$ are related via Darcy's law:
$$
u = - \n p - ( 0, \rho ).
$$
Here we set all physical constants to $1$. Eliminating the pressure $p$ from the relation above one obtains $u = T[\rho]$, where 
$$
m(\xi) = |\xi|^{-2} (\xi_1\xi_2, - (\xi_1)^2 ).
$$
We see that every frequency vector $\xi$ is a regular point for the symbol, and $\Rg(m) = \{z\in \R^2: |z + (0,1/2)| = 1/2 \}$, which clearly satisfies (ii). \thm{t:main} in this particular case was proved in \cite{spanish}. As noticed in \cite{spanish}  Darcy's law can be written in divergence form:
$$
\diver_x( -u_2-\rho, u_1) = 0
$$
and as a consequence the full system \eqref{ase}--\eqref{int} can be written as $\diver_{t,x} U=0$ with the appropriately defined $3$-by-$3$ matrix-valued function $U$. 

In 3D, the symbol becomes
$$
m(\xi) = |\xi|^{-2} \left(  \xi_1 \xi_3 ,  \xi_2 \xi_3, -\xi_1^2- \xi_2^2 \right).
$$
The range here is the sphere of radius $1/2$ centered at $(0,0,-1/2)$. So, the theorem applies.
We see on this example that the relation $u = T[\th]$ cannot be written in divergence form. This necessitates the departure from the traditional hyperbolic system setup as in \cite{spanish,ds1}.
\end{example}

\begin{example} This next example arises in the context of magnetostrophic turbulence inside the Earth's core as proposed by Moffatt \cite{moffatt}. The scalar $\th$ depends on three spacial coordinates and represents the buoyancy coefficient. The symbol of the operator $T$ was derived in \cite{vlad} and is given by
$$
m = ( M_1, M_2, M_3 ),
$$
\begin{align*}
M_1(\xi) &= \frac{  \xi_2 \xi_3 |\xi|^2 + \xi_1\xi_2^2 \xi_3}{  \xi_3^2 |\xi|^2 +  \xi_2^4} \\
M_2(\xi) &= \frac{ - \xi_1 \xi_3 |\xi|^2 + \xi_2^3  \xi_3}{  \xi_3^2 |\xi|^2 +  \xi_2^4} \\
M_3(\xi) &= \frac{- \xi_2^2(\xi_1^2 + \xi_2^2)}{  \xi_3^2 |\xi|^2 +  \xi_2^4}. \\
\end{align*}
In this case the symbol is not bounded on the sphere $\S^{2}$, and hence $T$ is not bounded on $L^2$. In fact, as shown in \cite{vlad}, $T$ acts as an operator of order $-1$, thus sending $L^2$ into $H^{-1}$. However, one can easily check that the hypotheses of \thm{t:main} are satisfied as there are plenty of regular points away from the $\xi_1$-axis, on which the symbol spans all of $\R^3$. It is notable that the corresponding viscous equation
$$
\th_t + u \cdot \n \th = \nu \D \th,
$$
as shown in \cite{vlad} using the Di Georgi iteration, admits global smooth solutions, even though the equation is critical to the use of energy methods. 
\end{example}

Finally we note that the symbol of the classical surface quasi-geostrophic equation is odd:
$$
m(\xi) = i |\xi|^{-1} (-\xi_2, \xi_1).
$$
So, this case is out of the scope of \thm{t:main} (see also the discussion in \cite{spanish}).

\section{Wave cone}
Following the traditional approach we write
\begin{align}
\diver_{t,x} U & = 0 \label{rel1} \\
U &= \begin{pmatrix}
     \th & q   \\
       0 &  u
\end{pmatrix}\\
u & = T[\th]. \label{rel3}
\end{align}
It is clear that any weak solution to \eqref{rel1}--\eqref{rel3} that belongs pointwise to a bounded subset of
 $$
K = \left\{ \begin{pmatrix}
      \th &  \th u  \\
      0&  u
\end{pmatrix}: \th \in \R, u\in \R^n \right\},
$$
yields a weak solution to our original system \eqref{ase}--\eqref{int}.

We first construct localized oscillatory solutions to \eqref{rel1} -- \eqref{rel3} along certain allowed directions. For this purpose we consider the set 
$$
\L = \left\{ \begin{pmatrix}
      \th &  q  \\
       0 &  \th m(\xi)
\end{pmatrix}:  \th \neq 0,  \xi \in \S^{n-1} \text{ is regular }, q \in \R^n\right\}.
$$
From now on we will repeatedly use the following notation for the space-time domain $\O_T = (0,T) \times \T^n$, and the space-time variable $y = (t,x)$.
\begin{lemma} \label{l:osc} 
Let us fix an open subdomain $\Omega \ss \O_T$, $\e >0$, and $\l\in (0,1)$. Then for every element of the wave cone $L \in \L$ there is a sequence of matrix-functions $\{Z_k\}_{k=1}^\infty \ss C_0^\infty(\O_T)$ satisfying the relaxed problem \eqref{rel1} -- \eqref{rel3} and such that 
\begin{enumerate}
  \item $Z_k \ra 0$ weakly$^*$;
  \item $\sup_{y \not \in \O} |Z_k(y)| < \e$, for all $k \in \N$;
  \item $\sup_{y \in \O} \dist\{ Z_k(y), [-(1-\l)L,\l L] \} <\e$;
  \item $\left| \{ y \in \O:  Z_k(y) \in B_\e(\l L)\} \right| > (1-\l)(1-\e)|\O|$;
  \item $\left| \{y \in \O:  Z_k(y) \in B_\e(-(1-\l) L) \} \right| > \l (1-\e)|\O|$;
\end{enumerate}
\end{lemma}
Notice that the sequence $Z_k$ is only supported in $\O$ up to a small error, yet it does lie within our general temporal limits $(0,T)$.
\begin{proof}
Let us fix our notation for $L$:
$$
L=\begin{pmatrix}
      \th_0 &  q_0  \\
       0 &  \th_0 m(\xi_0)
\end{pmatrix},
$$
for some $\th_0 \neq 0$,  regular $\xi_0 \in \S^{n-1}$, and  $q_0 \in \R^n$. Our primary goal at this point is to construct a divergence-free vector field $V(y)$ that oscillates around the first row of $L$, i.e. vector $(\th_0,q_0)\in \R^{n+1}$. 
Let us write $\xi_0 = \lan \k^0_1,...,\k^0_n \ran$ and assume for definiteness that $\k^0_1 \neq 0$. We will find $V$ as a linear combination of divergence-free fields of the form
\begin{align}
\calD_1(\psi) & = \lan \D_x \psi, - \p_{t,x_1} \psi, \ldots, - \p_{t,x_n} \psi \ran \\
\calD_i(\phi) & = \lan 0, -\p_{x_i} \phi, 0, \ldots, 0, \p_{x_1} \phi,0, \ldots \ran, \quad i = 2,...,n,
\end{align}
where the potentials $\phi$ and $\psi$ are smooth scalar functions to be determined later. It is easily seen that regardless of their choice, the fields $\calD_i$, $i=1,...,n$ are divergence-free.

Let us fix the function $f= \l \chi_{[0,1-\l)} - (1-\l) \chi_{[1-\l, 1)}$ and extend it priodically from the unit interval $[0,1)$ to $\R$. Let us fix $\e_1>0$ and consider a $1$-periodic trigonometric sum $\tilde{f}$ obtained from $f$ via truncation of  the Fourier series of $f$. By choosing the truncation far enough we can ensure the following two conditions:
 \begin{align}
 \left|\{s: |\tilde{f}(s) - \l | <\e_1 \}\right| & > (1-\l)(1-\e_1) \label{proxf1}\\
 \left|\{s: |\tilde{f}(s) +1- \l | <\e_1\} \right| & > \l(1-\e_1). \label{proxf2}
 \end{align} 
Dropping tildas for notational convenience, we now have a smooth function $f$ with finite Fourier support, satisfying \eqref{proxf1}-- \eqref{proxf2}. We also set  $1$-periodic $F$ such that $F'' = f$. 

Next, we fix another small parameter $\e_2$, a compactly embedded subdomain $\O' \ss \O$ with $| \O' | > (1-\e_2) |\O|$, and a localization function $h \in C_0^\infty( \O)$, $0 \leq h \leq 1$ such that $h = 1$ on $\O'$. We denote by $\tilde{h}$ the function obtained from $h$ via truncating the Fourier series of $h(t,\cdot)$, uniformly for each $t \in (0,T)$ far enough so that 
\begin{align}
-\e_2 \leq  \tilde{h} \leq 1+ \e_2 &\\
\left|\{ y\in \O: | \tilde{h}(y) - 1| < \e_2 \} \right| & > (1-\e_2) |\O| \\
\sup_{y \not \in \O} \{ |\tilde{h}(y) |,|D \tilde{h}(y) |,|D^2\tilde{h}(y) | \}& < \e_2.
\end{align}
In addition, since the truncation is performed only with respect to spacial dependence, we also have
\begin{equation}\label{}
\supp(\tilde{h}) \ss \O_T.
\end{equation}
We will drop tildas from $h$ below as well.

For any small scale parameter $\d >0$ we find a frequency vector $\xi = \lan \k_1,...,\k_n \ran$ such that $\xi/\d \in \frac{1}{2\pi} \Z^n$ and
\begin{equation}\label{xid}
| \xi  - \xi_0 | \leq \d \frac{\sqrt{n}}{4\pi}.
\end{equation}
We can now define
\begin{align}
\psi(y) &= \d^2 F((d_1 t + x \cdot \xi)/\d) h(y) \label{}\\
\phi(y) & = \d F'((d_1 t + x \cdot \xi)/\d) h(y)\label{}.
\end{align}
Here $d_1\in \R$ is to be determined later. Finally, we let
\begin{equation}\label{}
V = \th_0 \calD_1(\psi) + \sum_{i=2}^n d_i \calD_i(\phi).
\end{equation}
Directly from the definitions we obtain
\begin{align}
V = [ \lan \th_0 |\xi|^2, 0, \ldots, 0 \ran & + d_1\lan 0, -\th_0 \k_1,\ldots, - \th_0 \k_n \ran  \label{d1}\\
& + \sum_{i=2}^n d_i \lan 0, -\k_i,\ldots, \k_1, \ldots, 0\ran ] f h + O(\d) \label{dn}
\end{align}
Let us notice that
$$
\det \begin{pmatrix}
      -\th_0 \k_1& -\k_2 & \ldots & - \k_n   \\
      -\th_0 \k_2 &  \k_1& \ldots & 0 \\
                \vdots         &    \vdots   &   \ddots &  \vdots\\
      -\th_0 \k_n & 0  &  \ldots & \k_1                   
\end{pmatrix} = - \th_0 \k_1^{n-2} | \xi|^2 \neq 0.
$$
So, one can find an $n$-tuple $(d_1,\ldots,d_n)$ such that the corresponding vectors in \eqref{d1} -- \eqref{dn} add up to $\lan 0, q_0 \ran$. Along with \eqref{xid} this ensures that $V$ has the following form
$$
V = \lan \th_0, q_0 \ran fh + O(\d).
$$
Let us now denote by $\th$ the first component of $V$, or more explicitely,
\begin{equation}\label{th}
\th = \D_x \psi = \th_0 |\xi|^2 f h + \d F'\lan \xi \cdot \n_x h \ran + \d^2 F \D_x h.
\end{equation}
Let us define $u = T[\th]$. According to our construction above $\th(t,x)$ has a finite Fourier support in spacial frequencies uniformly for all $t \in (0,T)$. More procisely, there is $R>0$ such that 
$$
\supp( \widehat{\th(t,\cdot)} ) \ss B_R(\xi_0/\d) \cup B_R(- \xi_0/\d).
$$
Since $\xi_0 \in \S^{n-1}$ is a regular point for the symbol $m$ there is an open neighborhood $W\ss \S^{n-1}$ of $\xi_0$ where $m \in C^1$. It is clear that for $\d$ small enough the radial projection of $\supp( \widehat{\th(t,\cdot)} )$ will land into $W\cup(-W)$ for all $t\in (0,T)$. Since, in addition, $m$ is even and $0$-homogeneous, we can perform the following standard computation
\begin{align}
u(t,x) & = \sum_{k \in \Z^n} e^{i k\cdot x} m(k)  \widehat{\th(t,\cdot)}(k)  \label{}\\
& =  \sum_{\substack{|\d k - \xi_0| < \d R \\ |\d k + \xi_0| < \d R} } e^{i k\cdot x} (m(\d k) - m(\xi_0))  \widehat{\th(t,\cdot)}(k)  + m(\xi_0) \th(t,x) \label{}
\end{align}
and thus in view of \eqref{th}
\begin{equation}\label{}
\|u  - \th_0 m(\xi_0) f h \|_\infty \lesssim \d.
\end{equation}
It is now clear that choosing $\e_1,\e_2>0$ small enough and sending $\d \ra 0$, the matrix-functions $
Z(\d) = \left(\begin{smallmatrix}
      &V    \\
      0  &u  
\end{smallmatrix}\right)
$ satisfy all the properties listed in the lemma.
\end{proof}

In the sequel we will actually use a slightly different variation of \lem{l:osc}. Let us suppose that we have two matrices $A_1$ and $A_2$ such that $L = A_2-A_1 \in \L$, and let $A^* \in (A_1,A_2)$ be another matrix on the interval joining $A_1$ and $A_2$. Then, there exists $\l \in (0,1)$ such that $A^* = \l A_1 + (1-\l)A_2$. So, we have
$A^* + \l L = A_2$ and $A^* - (1-\l) L = A_1$. By applying \lem{l:osc} to $L$ and $\l$,  we find a sequence $Z_k$ with the properties (1) and (2) above and
\begin{itemize}
  \item[(3*)] $\sup_{y \in \O} \dist\{A^*+ Z_k(y), [A_1,A_2] \} <\e$;
  \item[(4*)] $|\{ y \in \O: A^*+ Z_k(y) \in B_\e(A_2) \}| > (1-\l)(1-\e)|\O|$;
  \item[(5*)] $|\{ y \in \O: A^*+ Z_k(y) \in B_\e(A_1) \}| > \l (1-\e)|\O|$;
\end{itemize}
We can see that every element of the shifted  sequence spends nearly all time in the vicinities of the end points of the interval.

One last property we can retrieve from the argument above is the localization of the Fourier support of the sequence near the line $[-\xi_0,\xi_0]$. Let $\bar{\xi} = \xi/|\xi|$. We have
\begin{itemize}
  \item[(6*)] For every open neighborhood of $\xi_0$, $W \ss \S^{n-1}$ one can ensure the inclusion
  $$
     \overline{\supp \widehat{Z_k(t,\cdot)}} \ss W \cup (-W),
   $$
   for all $t \in (0,T)$ and $k\in \N$.
\end{itemize}

\section{Geometric considerations}
To every regular point $\xi \in \S^{n-1}$ we can associate an open neighborhood $\xi \in W \ss \S^{n-1}$ such that $m$ is $C^1$ on a neighborhood of $\bar{W}$, and $\rest{m}{W}$ is an immersion. So, the image $S=m(W)$ is a $C^1$-hypersurface of $\R^n$.  In view of condition (ii) of the theorem, we can find two such hypersurfaces $S_1 = m(W_1)$ and $S_2 = m(W_2)$ so that the interval $I$ joining points $m_1 = m(\xi_1)$ and $m_2 = m(\xi_2)$ intersects $S_1$ and $S_2$ transversally. Let us fix two distinct points $a,q_0 \in I$ which are different from $m_1$ and  $m_2$. By narrowing down the surfaces $S_i$'s around $m_i$'s if necessary we can ensure that there exists an open ball $B_{\d_0}(q_0)$ centered at $q_0$ of radius $\d_0 >0$, not containing $a$, and such that every point $x \in B_{\d_0}(q_0)$ uniquely determines two other points $x_1\in S_1$ and $x_2\in S_2$ which lie on the line joining $a$ and $x$. We can also assume that $W_j$

Let us notice that under small perturbations of the "screens" $S_1$ and $S_2$ the projection points $x_1$ and $x_2$ are still uniquely defined from $x$. Specifically, let $u \in \R^n$ and $\th \in \R$ are so that $|u|, |\th| < \d_1$ for some small $\d_1>0$. Then any point $x \in B_{\d_0}(q_0)$ defines $x_1 \in u + (1-\th) S_1$ and $x_2 \in u + (1-\th)S_2$ is the manner decribed above. Moreover, the line joining $x_1$ and $x_2$ crosses the screens $u + (1-\th) S_j$ transversally at angles bounded away from zero.

Let us introduce into consideration the restricted cone:
$$
\L_W = \left\{ \begin{pmatrix}
      \th &  q  \\
       0 &  \th m(\xi)
\end{pmatrix}:  \th \neq 0,  \xi \in W_1 \cup W_2, q \in \R^n\right\}.
$$
Denote 
$$A_0 = \begin{pmatrix}
      0 & q_0   \\
     0 &  0
\end{pmatrix}.$$
 We now proceed with the following lemma.

\begin{lemma}\label{l:t4} There exists a $\d >0$ such that every matrix $A$ satisfying
\begin{equation}\label{}
|A - A_0 | <\d
\end{equation}
can be represented as a convex combination of four matrices $A = \sum_{j=1}^4 \l_j T_j(A)$, with all $\l_j \in (0,1)$, and $T_j(A) \in K$ of the form
\begin{align*}
T_1(A) &= \begin{pmatrix}
      1 &    x_1 \\
       0 & x_1 
\end{pmatrix}
&T_3(A) &= \begin{pmatrix}
      -1 &    -y_1 \\
       0 & y_1 
\end{pmatrix}\\
T_2(A) &= \begin{pmatrix}
      1 &    x_2 \\
       0 & x_2 
\end{pmatrix}
&T_4(A) &= \begin{pmatrix}
      -1 &    -y_2 \\
       0 & y_2 
\end{pmatrix}
\end{align*}
Moreover, $T_j(A) - A \in \L_W$ for all $j$, and the maps $A \ra x_i(A)$, $A \ra y_i(A)$ are open as maps from $\R^{2n+1}$ into $\R^n$.  
\end{lemma}
\begin{proof}
Let $A =\left( \begin{smallmatrix}
     \th & q   \\
     0 & u 
\end{smallmatrix}\right)$. One can represent $A$ as
$$
A = \frac{1+\th}{2} X + \frac{1-\th}{2} Y,
$$
where
$$
X = \begin{pmatrix}
     1 & x   \\
     0 & x 
\end{pmatrix}, \quad
Y = \begin{pmatrix}
     -1 & -y   \\
     0 & y 
\end{pmatrix}
$$
and
$$
x = \frac{u+q}{1+\th}, \quad y = \frac{u-q}{1-\th}.
$$
Notice that if $\d>0$ is small enough, then $x \in B_{\d_0}(q_0)$ and $y \in B_{\d_0}(-q_0)$.  According to the scheme set forth before the lemma we can find four uniquely defined points $x_j\in u + (1- \th)S_j$ and (by the symmetry) $y_j \in u - (1+\th) S_j$. Let us verify that these points, and the corresponding $T_j$'s satisfy the conclusions of the lemma. First, the facts $A \in \conv\{T_j(A)\}_{j=1}^4$ and $T_j(A)\in K$ follow directly from the construction. Second, there exists $\xi_1 \in W_1$ such that $x_1 = u + (1-\th) m(\xi_1)$. Then 
$$
T_1(A) - A = \begin{pmatrix}
      1 - \th &   u + (1-\th) m(\xi_1) - q \\
      0 & (1-\th) m(\xi_1) 
\end{pmatrix} \in \L_W.
$$
The other inclusions $T_j(A) - A \in \L_W$, $j=2,3,4$, are verified similarly.

Let us prove that the map $A \ra x_1(A)$ is open, the other three follow in the same fashion. To this end, let us fix an
$$
A' = \begin{pmatrix}
     \th' &  q'  \\
       0 & u' 
\end{pmatrix} \in B_\d(A).
$$
Then $x' = \frac{u' + q'}{1+\th'} \in B_{\d_0}(q_0)$, and $x_1(A') \in u' + (1-\th')S_1$. We will keep $\th'$ fixed at all time. Varying $q'$ in its small neighborhood it is clear that the point $x_1(A')$ describes a relatively open subset of the screen $u' + (1-\th')S_1$. Denote it $Q_1(u')$. Next observe that there is a small $\d_3>0$ such that for all $u''$ with $|u''-u'|<\d_3$ the distance of $x(A'')$ to the boundary of $Q_1(u'')$ remains bounded away from zero. It remains to vary $u''$ in the direction transversal to the surface $(1-\th')S_1$ to obtain an open set containing $x_1(A')$.
\end{proof}

Let us fix $\d >0$ as in the previous lemma and define the sets
\begin{align}
\K &= \cup_{j=1}^4 T_j(B_\d(A_0)) \ss K  \label{}\\
\U &= \cup_{j=1}^4 \conv^{\L_W} \left\{ B_\d(A_0), T_j(B_\d(A_0))  \right\} \backslash \K.
\end{align}
Here we use the $\L_W$-convex hull of two sets defined by
$$
\conv^{\L_W} \left\{ E_1, E_2 \right\} = \left\{ \l A_1 + (1-\l)A_2: A_j \in E_j, A_1 - A_2 \in \L_W \right\}
$$

We will now use the openness of the maps defined in \lem{l:open} to show that $\U$ is an open subset of $\R^{2n+1}$.
\begin{lemma}\label{l:open} The set $\U$ is open in $\R^{2n+1}$.
\end{lemma}
\begin{proof}
Let $A \in \U$, and let us exclude the trivial case $A\in B_\d(A_0)$. So, there is $t\in(0,1)$, $A',A''\in B_\d(A_0)$ such that $A = t A' + (1-t)T_j(A'')$, and $T_j(A'') - A' \in \L_W$. Let us assume that $j=1$ for definiteness. Let
$$
A' = \begin{pmatrix}
      \th & q   \\
      0 & u 
\end{pmatrix}, \quad 
T_1(A'') = \begin{pmatrix}
      1&  x  \\
      0 & x 
\end{pmatrix}.
$$
Since the map $A'' \ra x(A'')$ is open for $\d x$ small enough we have
$T_1(A'')+ \d T \in T_1(B_\d(A_0))$, where $\d T =\left( \begin{smallmatrix}
     0 & \d x   \\
     0 &  \d x
\end{smallmatrix}\right)$. Then for $\d A' = \left(\begin{smallmatrix}
     0 & \d q   \\
      0&  \d x
\end{smallmatrix}\right)$ we still have $T_1(A'') + \d T - A' - \d A' \in \L$ and $A' + \d A' \in B_\d(A_0)$. Thus, for all small $\d x, \d q, \d t$ we have
$$
A+\d A = (t+ \d t)(A' + \d A') + (1- t -\d t)(T_1(A'') + \d T) \in \U.
$$
We have from the above 
$$
\d A = \d t A' + t \d A' + \d t \d A' + (1-t) \d T - \d t T_1(A'') - \d t \d T.
$$
Let us consider $\d A$ as a map from $\R^{2n+1}$ into itself. The Jacobian of this map at the origin is
$$
J_{2n+1} = \det \begin{pmatrix}
      \th-1 & q-x& u-x   \\
      0 &  (1-t)I_n & I_n\\
      0& 0& t I_n
\end{pmatrix} = (\th-1)(1-t)^nt^n \neq 0.
$$
Thus, by the inverse function theorem, $\d A$ fills an open neighborhood of $0$.
\end{proof}

\section{Functional considerations}

We now use \lem{l:osc} to construct solutions within the set $\U$ oscillating around a point $A$ and having the amplitude of oscillation comparable to the distance of $A$ to the target constraint set  $\K$.

\begin{lemma} \label{l:func} 
For any $\Omega \ss \O_T$, $\e >0$, and $A \in \U$ there exists a sequence $\{Z_k\}_{k=1}^\infty \ss C_0^\infty(\O_T)$ satisfying the relaxed problem \eqref{rel1} -- \eqref{rel3} and such that 
\begin{enumerate}
  \item $Z_k \ra 0$ weakly$^*$;
  \item $\sup_{y \not \in \O} |Z_k(y)| < \e$, for all $k \in \N$;
  \item $A + Z_k(y) \in \U$, for all $y \in \O$ and $k \in \N$;
  \item $\left| \{ y \in \O:  | Z_k(y)| \geq c_0 \dist\{A,\K\} \} \right| > c_1|\O|$,
  \item $\overline{ \supp{ \widehat{Z_k(t,\cdot)}  }} \ss W_1 \cup (-W_1) \cup W_2 \cup (-W_2)$, for all $k\in \N$, and all $t\in(0,T)$,
\end{enumerate}
where $c_0,c_1>0$ are some absolute constants.
\end{lemma}
\begin{proof}
Let us assume first that $A \in B_\d(A_0)$. Let $T_j(A)$ be as in 
\lem{l:t4}. We have $A = \sum_{j=1}^4 \l_j T_j(A)$ for some $\l_j\in (0,1)$ and $\sum_{j=1}^4 \l_j = 1$. Clearly, moving $T_j$'s slightly closer to $A$, i.e. letting $T^s_j = (1-s)T_j(A) + s A$ for some small $s>0$ we obtain $T^s_j \in \U$, and still $A = \sum_{j=1}^4 \l_j T^s_j$. Also notice that $| A - T^s_j | =  (1-s) |A - T_j(A)| \geq (1-s) \dist\{ A, \K\}$, for all $j=1,2,3,4$. Now let us move the segments $[A,T^s_j]$ around as follows. Let us define for $i=1,2,3,4$ 
$$
A_i = A + \eta \sum_{j=1}^i \l_j (T^s_j - A).
$$
Notice that $A_4 = A$, so the sequence closes and we have four distinct points $A_1, A_2, A_3, A_4$. If $\eta>0$ is chosen small enough then clearly all four belong to $B_\d(A_0)$. Let $T_i = T^s_i + A_{i-1} - A$. We now obtained the classical $T4$-configuration, where $T_i - A_{i-1} \in \L_W$, and 
$$
A_i = (1-\eta \l_i)A_{i-1} + \eta \l_i T_i,
$$
for all $i=1,2,3,4$ (here, with a little abuse of notation, we use the cyclic notation where $A_0 = A$). Again, if $\eta$ is small enough, $T_i$'s become perturbations of $T^s_i$'s and hence 
\begin{equation}\label{away2}
|T_i - A| \geq \frac{1}{2} \dist\{ A, \K\}.
\end{equation}
One last requirement on $\eta$ is 
\begin{equation}\label{eta}
(1-\eta)^3 > 1/2.
\end{equation}

We construct the sequence $Z_k$ based on an iterative procedure in the spirit of \cite{}. There will be only finitely many steps in the process. The number of steps we need, denoted $N$, is divisible by $4$ and is such that 
\begin{equation}\label{N}
(1-\eta)^{N-4} < 1/2.
\end{equation}
So, let us fix such an $N$ from now on. Our next parameter $\e>0$ will be determined later and will depend on various other parameters and conditions. So, we will keep it arbitrary at the moment. 

Our starting point is $A_4=A$ itself. We apply the variant of \lem{l:osc}  stated in the remark following its proof we find a sequence $Z_{k_1} \ra 0$ weakly$^*$, with condition (5) ensured by (6*), and  such that 
\begin{equation}\label{}
\sup_{y\not \in \O}|Z_{k_1}(y)| < \e,
\end{equation}
while the sets
\begin{align}
\O_{k_1}^A &= \{ y \in \O: A_4 + Z_{k_1} \in B_\e(A_3) \}   \label{}\\
\O_{k_1}^T &= \{ y \in \O: A_4 + Z_{k_1} \in B_\e(T_4) \}      \label{}
\end{align}
have measures satisfying
\begin{align}
| \O_{k_1}^A| &> (1-\eta \l_4)(1-\e) |\O| > (1-\eta )(1-\e) |\O| \label{}\\
| \O_{k_1}^T| &> \eta \l_4 (1-\e) |\O|.   \label{}
\end{align}
For each $k_1\in \N$ we constract a new sequence $Z_{k_1k_2}$ based on the set $\O_{k_1}^A$ and oscillating around $A_3$, and with Fourier supports as in (5). We find
\begin{equation}\label{}
\sup_{y\not \in \O_{k_1}^A}|Z_{k_1k_2}(y)| < \e,
\end{equation}
and the sets
\begin{align}
\O_{k_1k_2}^A &= \{ y \in \O_{k_1}^A: A_3 + Z_{k_1k_2} \in B_\e(A_2) \}   \label{}\\
\O_{k_1k_2}^T &= \{ y \in \O_{k_1}^A: A_3 + Z_{k_1k_2} \in B_\e(T_3) \}      \label{}
\end{align}
have measures satisfying
\begin{align}
| \O_{k_1k_2}^A| &>  (1-\eta )^2(1-\e)^2 |\O| \label{}\\
| \O_{k_1k_2}^T| &> \eta \l_3 (1-\e)^2(1-\eta) |\O|.   \label{}
\end{align}
Continuing in the same fashion $N$ times we obtain a collection of $N$  weakly$^*$-null sequences $Z_{k_1}, Z_{k_1k_2},\ldots Z_{k_1\ldots k_N}$ such that 
\begin{equation}\label{}
\sup_{y \not \in \O^A_{k_1\ldots k_{i-1}}} |Z_{k_1\ldots k_i}| <\e,
\end{equation}
and letting
\begin{align}
\O_{k_1\ldots k_i}^A &= \{ y \in \O^A_{k_1\ldots k_{i-1}}: A_{5-i} + Z_{k_1\ldots k_i} \in B_\e(A_{4-i}) \}   \label{}\\
\O_{k_1\ldots k_i}^T &= \{ y \in \O_{k_1\ldots k_{i-1}}^A: A_{5-i} + Z_{k_1\ldots k_i} \in B_\e(T_{5-i}) \}      \label{}
\end{align}
we have the estimates
\begin{align}
| \O_{k_1\ldots k_i}^A| &>  (1-\eta )^i(1-\e)^i |\O| \label{}\\
| \O_{k_1\ldots k_i}^T| &> \eta \l_{5-i} (1-\e)^{i}(1-\eta)^{i-1} |\O|.  \label{omegaT}
\end{align}
Here again we are using the cyclic notation for $A$, $\l$, and $T$ modulo $4$. 
Let us now select diagonal weakly$^*$-null subsequences $Z_{k_1(n)}, \ldots, Z_{k_1(n)\ldots k_N(n)}$ and denote
\begin{equation}\label{}
\calZ_n = Z_{k_1(n)} +  \ldots + Z_{k_1(n)\ldots k_N(n)}.
\end{equation}
We claim that $\calZ_n$ satisfies all the properties stated in the lemma. Condition (1) is immediate, and we have $\sup_{y\not\in \O} |\calZ_n(y)| < \e N$. With the appropriate choice of $\e$ we can ensure the desired smallness condition (2). The inclusion (3) also verifies directly. We have
\begin{align*}
A_4 &+ Z_{k_1}  \ss B_\e([A_3,T_4]),\\
\rest{(A_4& + Z_{k_1}+Z_{k_1k_2})}{\O^A_{k_1}} \ss B_\e(A_3)+Z_{k_1k_2} \ss B_{2\e}([A_2,T_3]) \\
\rest{(A_4 &+ Z_{k_1}+Z_{k_1k_2})}{\O\backslash \O^A_{k_1}} \ss B_{2\e}([A_3,T_4]),
\end{align*}
continuing in the same manner we obtain
\begin{equation}\label{embedZn}
A_4 + \calZ_n \ss B_{N\e}( \cup_{i=0}^3[A_i,T_{i+1}] ),
\end{equation}
which is a subset of $\U$ provided $\e$ is small enough.

Finally, condition (4) can be checked as follows. Let
\begin{equation}\label{}
\O^T(n) = \bigcup_{i=1}^N \O^T_{k_1(n)\ldots k_i(n)}.
\end{equation}
 Since all $\O^T$'s in the union are disjoint we obtain with the help of \eqref{omegaT}
\begin{equation}\label{}
|\O^T(n)| > |\O| \eta (1-\e) \sum_{i=0}^{N-1} \l_{4-i} ((1-\eta)(1-\e))^i.
\end{equation}
In view of \eqref{eta} and \eqref{N} we further obtain
\begin{multline*}\label{}
|\O^T(n)| > |\O| \eta (1-\e)\frac{1/2}{1-(1-\eta)(1-\e)}( \l_4  + \l_3 (1-\eta)(1-\e) \\ + \l_2 (1-\eta)^2(1-\e)^2 + \l_1 (1-\eta)^3(1-\e)^3 ) > \frac{1}{2}|\O|(1-\e) \frac{1}{1-\e+ \frac{\e}{\eta}}>\frac{1}{4}|\O|, 
\end{multline*}
provided $\e$ is chosen small enough. Yet if $y\in \O^T_{k_1\ldots k_i}\ss \O^A_{k_1\ldots k_{i-1}}$, then by construction
$$
\left| \sum_{j=i+1}^N Z_{k_1\ldots k_j}(y) \right| < \e (N-i).
$$
Moreover, for some $l(i),p(i)\in \{1,2,3,4\}$ we have
\begin{align*}
A_{l(i)}+ Z_{k_1\ldots k_i}(y) & \in B_\e(T_{p(i)}) \\
A_{l(i)+1}+ Z_{k_1\ldots k_{i-1}}(y) & \in B_\e(A_{l(i)}) \\
&\vdots \\
A_3 + Z_{k_1k_2}(y) &\in B_\e(A_2) \\
A_4 + Z_{k_1}(y)  &\in B_\e(A_3).  
\end{align*}
Thus,
\begin{multline*}
| \calZ_n + A_0 - T_{p(i)} | = |Z_{k_1} + A_0 - A_3 + A_3 + Z_{k_1k_2} - A_2+\ldots \\ \ldots +A_{l(i) + 1} + Z_{k_1\ldots k_{i-1}} - A_{l(i)}+A_{l(i)} +Z_{k_1\ldots k_{i}} - T_{p(i)} | \leq N \e,
\end{multline*}
and we have $|A_0 - T_{p(i)} | > \frac{1}{2} \dist\{A,\K\}$. In summary, condition (4) is fullfilled on the set $\O^T(n)$, for every $n\in \N$.

To finish the proof we now consider the case when $A \not \in B_\d(A_0)$. Since $A \in \U$ we have $A = t A' + (1-t)T_j(A'')$ for some $t\in (0,1)$, $A',A''\in B_\d(A_0)$, $j\in \{1,2,3,4\}$, and $T_j(A'') - A' \in \L$. If $t<1/2$ then we simply apply \lem{l:osc} to find a sequence $Z_k$ with 
$$
| \{ y\in \O: A+ Z_k \in B_\e(T_j(A'')) \} | > 1/2 |\O| (1-\e).
$$
Thus, with $\e$ small enough, for every point $y$ in the set above we have
$|Z_k(y)| > 1/2|T_j(A'') - A| \geq 1/2 \dist\{A,\K\}$. 
Let us assume now $t\geq 1/2$. We consider the $T4$-configuration as above associated with $A'$, but we start the construction of the sequence with $Z_{k_0}$ oscillating around $A$:
$$
|\O^A_{k_0}| = |\{ y \in \O: A+ Z_{k_0}(y) \in B_\e(A')\}| > t (1-\e) |\O|.
$$
We continue with the construction above based on the set  $\O^A_{k_0}$. We then arrive at the sequence $\calZ_n$ such that 
$$
| \{y \in \O: |\calZ_n(y)| > c_0 \min_{p}|A'-T_p| \} |> c_1 |\O|.
$$
It remains to notice that 
\begin{equation}\label{away}
|A' - T_j(A'')| \geq 1-\d,
\end{equation}
for any $A',A'' \in B_\d(A_0)$ and $j=1,2,3,4$. Thus, (4) is satisfied with some other $c_0,c_1>0$.
Finally, condition (5) is satisfied on each stem of the construction.
\end{proof}

\begin{remark}\label{contT4} From the argument above we can extract a more robust inclusion than that given by condition (3). 
We see that the $T4$-configuration constructed above from a point $A$, let's denote it $T4(A)$, is compactly embedded into $\U$ (in the case $A\not \in B_\d(A_0)$ this construction will contain an extra arm, but we still denote it $T4(A)$).
Let $\e_1 = \e_1(A)>0$ be such that $B_{\e_1}(T4(A)) \ss \U$. By our construction, $A+ Z_k(y) \in B_{\e_1/2}(T4(A))$ for all $y\in \O$ if $\e$ is chosen small enough. So, we additionally obtain a freedom to perturb $A$ in the above inclusion by a quantity indepent of $\e$. Specifically, if $A'\in B_{\e_1/2}(A)$, then $A'+ Z_k(y) \in B_{\e_1}(T4(A))\ss \U$.
 
Let us further notice that the value of $\e_1(A)$ can be chosen uniformly as $A$ varies in a compact subset $\calC \ss \U$. Indeed, the union $T4(\calC) = \cup_{A\in \calC} T4(A)$ can be constructed compactly embedded into $\U$. Thus, there exists $\e_1 = \e_1(\calC) >0$, such that $B_{\e_1}(T4(\calC)) \ss \U$ compactly.
\end{remark}

\section{Conclusion of the proof}

We now move to the final phase of the proof. This part of the argument is very similar to that of \cite{spanish,ds1}, although details will require some expansion due to the fact that we do not assume boundedness of $T$. To avoid this difficulty we use the extra control on the Fourier side as stipulated in condition (5) of \lem{l:func} which allows us to stay within the cone of frequencies where $T$ acts as a bounded operator. Details are provided below.

We define the space of solutions to the relaxed problem. Let $U \in X_0$ if
\begin{enumerate}
  \item $U \in C^\infty(\R \times \T^n)$;
  \item $U(y) = A_0$, for all $y\not \in \O_T$, and $\Rg(U) \ss \U$;
  \item $\overline{\supp\{ \widehat{U(t,\cdot)} \} \backslash \{0\} } \ss W_1 \cup (-W_1) \cup W_2 \cup (-W_2)$, for all $t\in(0,T)$;
  \item $\int_{\T^n} \th(t,x)dx = \int_{\T^n} u(t,x)dx = 0$, for all $t\in (0,T)$; 
  \item $U$ satisfies \eqref{rel1}--\eqref{rel3} in the classical sense.
\end{enumerate}
Clearly, $U \equiv A_0$ belongs to $X_0$.

\begin{lemma}\label{l:x0} Let $U_0 \in X_0$. There exists a sequence $\{U_k\}_{k=1}^\infty \ss X_0$ such that
 \begin{enumerate}
  \item $U_k \ra U_0$ weakly$^*$;
  \item $\liminf_{k\ra \infty}\| U_k\|_{L^2(\O_T)}^2 \geq \|U_0\|_{L^2(\O_T)}^2 + c_0 \int_{\O_T} \dist\{U_0(y), \K \}^2 dy$.
\end{enumerate} 
\end{lemma}
\begin{proof}
Let $\calC = \Rg(U_0)$. We have $\calC \ss \U$ compactly. Let $\e_1 = \e_1(\calC)>0$ be defined as in \rem{contT4}, so $B_{\e_1}(T4(\calC)) \ss \U$. By uniform continuity of $U_0$ there exists $r_0>0$ such that $U_0(B_{r_0}(y)) \ss B_{\e_1/4}(U_0(y))$, for all $y\in \O_T$. Let us find a finite collection of disjoint balls $\{B_{r_j}(y_j)\}_{j=1}^J$ with all $r_i < r_0$ such that
$$
2 \sum_{j=1}^J \dist\{ U_0(y_j), \K \}^2 |B_{r_j}(y_j)| >  \int_{\O_T} \dist\{U_0(y), \K \}^2 dy.
$$
For each $j$ we apply \lem{l:func} to $\O = B_{r_j}(y_j)$, $A = U(y_j)$, and $\e>0$ to be determined later. Thus, we find a sequence $Z_{j,k}$ with all the properties listed in the lemma. Moreover, in view of \rem{contT4}, we also have
$$
U_0(y) + Z_{j,k}(y) \in B_{\e_1/2}(T4(\calC)),
$$
for all $y\in B_{r_j}(y_j)$, and all $k\in \N$. Let $Z_k = \sum_{j=1}^J Z_{j, k}$, and $U_k = U_0 + Z_k$. Then for any $y\in B_{r_{j'}}(y_{j'})$ we have
$$
U_k(y) = U_0(y) + Z_{j',k}(y) + \sum_{j\neq j'} Z_{j,k}(y) \in B_{\e_1}(T4(\calC)),
$$
provided $\e < \e_1/(2J)$. If however $y$ does not belong to any of the balls, then 
$$
U_k(y) \in B_{\e J}(\calC) \ss B_{\e_1}(T4(\calC)).
$$
Clearly, outside of $\O_T$ all $Z_k$'s vanish and so $U_k = A_0$. Moreover the Fourier support condition (3) follows automatically from that of $Z_k$'s. This shows that $U_k \in X_0$, and so it remains to verify condition (2) of the lemma. We have since $Z_k \ra 0$ weakly in $L^2$:
$$
\liminf_{k\ra \infty}\| U_k\|_{L^2(\O_T)}^2  =\|U_0\|_{L^2(\O_T)}^2+  \liminf_{k\ra \infty}\| Z_k\|_{L^2(\O_T)}^2, 
$$
whereas for each $k\in \N$
\begin{align*}
\| Z_k\|_{L^2(\O_T)}^2 & \geq \sum_{j=1}^J \int_{B_{r_{j'}}(y_{j'})} | Z_{j,k}(y)|^2 dy - |\O|\e J  \\
& \geq c_1c_0 \sum_{j=1}^J|B_{r_{j'}}(y_{j'})| \dist^2\{U_0(y_j),\K\} - |\O|\e J \\
& \geq c \int_{\O_T} \dist^2\{U_0(y),\K\}  dy - |\O|\e J.
\end{align*}
So, with the appropriate choice of $\e>0$ we obtain (2).
\end{proof}
\def \cali {\mathbb{I}}
Let us consider $X = \overline{X_0}^{*}$, the closure of $X_0$ in the weak$^*$-topology of $L^\infty(\R \times \T^n)$. Let
$$
\cali: X \ra L^2_{loc}( \R \times \T^n)
$$
be the identity map, where $L^2_{loc}$ is endowed with the strong Frechet topology. Let $\f$ be a compactly supported mollifier. Then $\cali_\e U = U \ast \f_\e$ is a continuous map from $X$ into $L^2_{loc}$, and clearly $\cali_\e \ra \cali$ strongly. Thus, $\cali$ is a Baire-1 map, and as such its set of points of continuity is of second Baire category. We claim that any point of continuity $U$ satisfies $U(y) \in \K$ for a.e. $y\in \O_T$. Indeed, let $U_n \ra U$ weakly$^*$, and hence strongly in $L^2(\O_T)$, where each $U_n\in X_0$. Then let, according to \lem{l:x0}, $U_{n,k} \ra U_n$ weakly$^*$ and such that
$$
\liminf_{k\ra \infty}\| U_{n,k}\|_{L^2(\O_T)}^2 \geq \|U_n\|_{L^2(\O_T)}^2 + c_0 \int_{\O_T} \dist\{U_n(y), \K \}^2 dy.
$$
By taking a diagonal subsequence $U'_n$ which converges to $U$ strongly in $L^2(\O_T)$ we can pass to the limit in the above inequality and find
$$
\|U\|_2^2 \geq \|U\|_2^2 +  + c_0 \int_{\O_T} \dist\{U(y), \K \}^2 dy,
$$
which confirms our claim.

Being the $L^2$-strong limit of solutions to \eqref{rel1}--\eqref{rel3}, we see that $U$ satisfies \eqref{rel1} weakly. Moreover, the corresponding scalar and velocity components of $U'_n$, converge strongly in $L^2(\O_T)$:  $u_n \ra u$ and $\th_n \ra \th$. They, therefore, converge in $L^2(\T^n)$ for a.e. $t\in (0,T)$. Since in addition $\overline{\supp\{\hat{\th}_n, \hat{\th}, \hat{u}, \hat{u}_n\} }\ss 
W_1 \cup (-W_1) \cup W_2 \cup (-W_2)$ for a.e. $t$, we obtain by continuity of $T$ on such functions, that $u = T[\th]$. Hence, \eqref{rel3} holds. 

Although $U$ does not belong to $K$ outside $(0,T)$, this does not prevent the pair $(\th,u)$ to be a global weak solutions due to the fact that $\O_T$ has no spacial boundary outside $(0,T)$. To verify this let $\phi \in C_0^\infty(\R \times \T^n)$, then since $\th$ and $u$ vanish outside $(0,T)$, we have
\begin{multline*}
0 = \int_{\R \times \T^n} (\th \phi_t + q \cdot \n_x \phi ) dy =  \int_{(0,T) \times \T^n} (\th \phi_t + \th u \cdot \n_x \phi ) dy \\+ \int_{(0,T)^c \times \T^n} q_0 \n_x \phi dx dt =  \int_{\R \times \T^n} (\th \phi_t + \th u \cdot \n_x \phi ) dy.
\end{multline*}
This finishes the proof of \thm{t:main}.

Finally, we note that a more direct construction of solutions based on an approximation procedure is available. We refer the reader to \cite{spanish,ds1} for details.


\begin{thebibliography}{10}

\bibitem{const}
Peter Constantin.
\newblock Scaling exponents for active scalars.
\newblock {\em J. Statist. Phys.}, 90(3-4):571--595, 1998.

\bibitem{spanish}
Diego Cordoba, Daniel Faraco, and Francisco Gancedo.
\newblock Lack of uniqueness for weak solutions of the incompressible porous
  media equation.
\newblock 2010.
\newblock preprint.

\bibitem{ds1}
Camillo De~Lellis and L{\'a}szl{\'o} Sz{\'e}kelyhidi, Jr.
\newblock The {E}uler equations as a differential inclusion.
\newblock {\em Ann. of Math. (2)}, 170(3):1417--1436, 2009.

\bibitem{ds2}
Camillo De~Lellis and L{\'a}szl{\'o} Sz{\'e}kelyhidi, Jr.
\newblock On admissibility criteria for weak solutions of the {E}uler
  equations.
\newblock {\em Arch. Ration. Mech. Anal.}, 195(1):225--260, 2010.

\bibitem{eyink}
Gregory~L. Eyink and Katepalli~R. Sreenivasan.
\newblock Onsager and the theory of hydrodynamic turbulence.
\newblock {\em Rev. Modern Phys.}, 78(1):87--135, 2006.

\bibitem{vlad}
Susan Friedlander and Vlad Vicol.
\newblock Global well-posedness for an advection-diffusion equation arising in
  magneto-geostrophic dynamics.
\newblock 2010.
\newblock preprint.

\bibitem{gromov}
Mikhael Gromov.
\newblock {\em Partial differential relations}, volume~9 of {\em Ergebnisse der
  Mathematik und ihrer Grenzgebiete (3) [Results in Mathematics and Related
  Areas (3)]}.
\newblock Springer-Verlag, Berlin, 1986.

\bibitem{kms}
Bernd Kirchheim, Stefan M{\"u}ller, and Vladim{\'{\i}}r {\v{S}}ver{\'a}k.
\newblock Studying nonlinear pde by geometry in matrix space.
\newblock In {\em Geometric analysis and nonlinear partial differential
  equations}, pages 347--395. Springer, Berlin, 2003.

\bibitem{k41}
A.~N. Kolmogorov.
\newblock The local structure of turbulence in incompressible viscous fluids at
  very large reynolds numbers.
\newblock {\em Dokl. Akad. Nauk. SSSR}, 1941.

\bibitem{kr}
R.~H. Kraichnan.
\newblock Small-scale structure of a scalar field convected by turbulence.
\newblock {\em Phys. of Fluids}, II(5):945--953, 1968.

\bibitem{moffatt}
H.K. Moffatt.
\newblock Magnetostrophic turbulence and the geodynamo.
\newblock In {\em IUTAM Symposium on Computational Physics and New Perspectives
  in Turbulence}, pages 339--346. Springer, Dordrecht, 2008.

\bibitem{ms}
S.~M{\"u}ller and V.~{\v{S}}ver{\'a}k.
\newblock Convex integration for {L}ipschitz mappings and counterexamples to
  regularity.
\newblock {\em Ann. of Math. (2)}, 157(3):715--742, 2003.

\bibitem{onsager}
L.~Onsager.
\newblock Statistical hydrodynamics.
\newblock {\em Nuovo Cimento (9)}, 6(Supplemento, 2(Convegno Internazionale di
  Meccanica Statistica)):279--287, 1949.

\bibitem{scheffer2}
V.~Scheffer.
\newblock Regularity and irregularity of solutions to nonlinear second order
  elliptic systems of partial differential equations and inequalities.
\newblock 1974.
\newblock Dissertation, Princeton University, (unpublished).

\bibitem{scheffer1}
Vladimir Scheffer.
\newblock An inviscid flow with compact support in space-time.
\newblock {\em J. Geom. Anal.}, 3(4):343--401, 1993.

\bibitem{shn1}
A.~Shnirelman.
\newblock On the nonuniqueness of weak solution of the {E}uler equation.
\newblock {\em Comm. Pure Appl. Math.}, 50(12):1261--1286, 1997.

\bibitem{shn2}
A.~Shnirelman.
\newblock Weak solutions with decreasing energy of incompressible {E}uler
  equations.
\newblock {\em Comm. Math. Phys.}, 210(3):541--603, 2000.

\bibitem{shv}
R.~Shvydkoy.
\newblock Lectures on the {O}nsager conjecture.
\newblock {\em Discrete and Continuous Dynamical Systems - Series S},
  3(3):473--496, 2010.

\bibitem{spring}
David Spring.
\newblock {\em Convex integration theory}, volume~92 of {\em Monographs in
  Mathematics}.
\newblock Birkh\"auser Verlag, Basel, 1998.
\newblock Solutions to the $h$-principle in geometry and topology.

\bibitem{tartar}
L.~Tartar.
\newblock Compensated compactness and applications to partial differential
  equations.
\newblock In {\em Nonlinear analysis and mechanics: {H}eriot-{W}att
  {S}ymposium, {V}ol. {IV}}, volume~39 of {\em Res. Notes in Math.}, pages
  136--212. Pitman, Boston, Mass., 1979.

\end{thebibliography}

\end{document}